\documentclass[11pt]{article}
\usepackage{amsmath, amsfonts, amsthm, amssymb, multicol}
\usepackage{graphicx}
\usepackage{float}
\usepackage{verbatim}

\allowdisplaybreaks

\usepackage[square,sort,comma,numbers]{natbib}
\usepackage{fancyhdr}
 
\numberwithin{equation}{section}

\newtheorem{thm}{Theorem}[section]
\newtheorem{lma}[thm]{Lemma}
\newtheorem{cor}[thm]{Corollary}
\newtheorem{defn}[thm]{Definition}

\newtheorem{prop}[thm]{Proposition}

\renewcommand{\epsilon}{\varepsilon}
\newcommand{\eps}{\varepsilon}
\newcommand{\rr}{\mathbb{R}^2}
\newcommand{\rd}{\mathbb{R}^n}

\renewcommand{\geq}{\geqslant}
\renewcommand{\leq}{\leqslant}

\newcommand{\ubd}{\overline{\dim}_{\textup{B}}}

\newcommand{\ad}{\dim_{\mathrm{A}} }

\newcommand{\hd}{\dim_{\mathrm{H}}  }

\newcommand{\pd}{\dim_{\mathrm{P}}  }

\title{ \vspace{-20mm} A nonlinear projection theorem for Assouad dimension \\ and applications}

\author{Jonathan M. Fraser}

\begin{document}


\maketitle

\begin{abstract}
We prove a general nonlinear projection theorem for Assouad dimension.  This theorem has several applications including to distance sets,  radial projections, and sum-product phenomena.  In the setting of distance sets we are able to completely resolve the planar distance set problem for Assouad dimension, both dealing with the awkward  `critical case' and providing sharp estimates for sets with Assouad dimension less than 1.   In the higher dimensional setting we connect the problem to the dimension of the set of exceptions in a related (orthogonal) projection theorem. We also obtain results on pinned distance sets and our results still hold when the distances are taken with respect to a sufficiently curved norm.    As another application we prove a  radial projection theorem for Assouad dimension with sharp estimates on the Hausdorff dimension of the exceptional set.
\end{abstract}

\section{Introduction}

How dimension behaves under projection is a well-studied and important problem in geometric measure theory with many varied applications.  The classical setting is to relate  the Hausdorff dimension of  a set $F \subseteq \rd$ with the Hausdorff dimension of $\pi_V(F)$ for generic $V \in G(n,m)$.  Here and throughout $G(n,m)$  denotes the Grassmannian manifold consisting of $m$-dimensional subspaces of $\rd$ and $\pi_V$ denotes orthogonal projection from $\rd$ to $V \in G(n,m)$.  We write $\hd E$ for the Hausdorff dimension of a set $E$.  The seminal Marstrand-Mattila projection theorem states that for Borel sets $F \subseteq \rd$
\begin{equation} \label{mmproj}
\hd \pi_V (F) = \min\{\hd F, m\}
\end{equation}
for almost all $V \in G(n,m)$.  Here `almost all' is with respect to the Grassmannian measure, which is the appropriate analogue of $m(n-m)$-dimensional Lebesgue measure on $G(n,m)$.   The planar case of this result goes back to Marstrand's 1954 paper \cite{marstrand} and the general case was proved by Mattila \cite{mattilaproj}.  This result has inspired much work in geometric measure theory, fractal geometry, harmonic analysis, ergodic theory and many other areas.

This paper is concerned with the Assouad dimension, which is  a well-studied notion of dimension with key applications in embedding theory, quasi-conformal geometry and fractal geometry.  The analogue of the Marstrand-Mattila projection theorem for Assouad dimension was proved in \cite[Theorem 2.9]{fraserdist}, the planar case having been previously established by Fraser and Orponen \cite{FraserOrponen}.  We write $\ad E$ for the Assouad  dimension of a set $E$. The result is that for any non-empty set $F \subseteq \rd$
\begin{equation} \label{assouadprojection}
\ad \pi_V( F) \geq  \min\{\ad F, m\}
\end{equation}
for almost all $V \in G(n,m)$.  An interesting feature of this result is that the inequality cannot be replaced by equality in general.  This latter fact  was proved in \cite{FraserOrponen} and  in \cite{kaenmakiproj} it was proved that, apart from satisfying \eqref{assouadprojection} almost surely, the behaviour of $\ad \pi_V( F)$ can be very wild.  Our projection theorems will share this phenomenon and we make no further mention of this.  

We are concerned with parameterised families of \emph{nonlinear} projections, rather than the orthogonal projections $\pi_V$.  Our treatment and exposition   takes some inspiration from the nonlinear projection theorems of Peres and Schlag \cite{peresschlag}, which are primarily in the setting of Hausdorff dimension of sets and measures.   The work of Peres and Schlag has proved influential, with the concept of transversality at the centre.  Their general nonlinear projection theorems have applications in several  areas including radial projections, distance sets, Bernoulli convolutions, sumsets, and many other `nonlinear' problems.  Our main result, Theorem \ref{thmmain}, is a general nonlinear projection theorem for Assouad dimension, and this too has many applications.  Most strikingly to distance sets, where we are able to  completely resolve the planar distance set problem for Assouad dimension, see Theorem \ref{distmain0}. Specifically, we prove  that the Assouad dimension of the distance set of a set $F$ in the plane is at least $\min\{\ad F, 1\}$.  In the higher dimensional setting we connect the problem to the dimension of the set of exceptions to \eqref{assouadprojection}, see Theorem \ref{distmain}.  We also obtain results for  pinned distance sets and for distance sets where the distances are taken with respect to a `sufficiently curved' norm.  Our proofs use tools from geometric measure theory, such as the theory of weak tangents \cite{mackaytyson, anti3}; fractal geometry, such as Orponen's projection theorem for Assouad dimension \cite{orponenassouad} and transversality; and also  differential geometry, with linearisation the underlying principle.

 For background on fractal geometry, including Hausdorff dimension and the dimension theory of projections, see the books \cite{falconer,mattila} and the recent survey articles on projections \cite{FalconerFraserJin, MattilaSurvey}. For background on the Assouad dimension, see the books \cite{fraserbook,robinson}, and for recent results on the Assouad dimension of orthogonal projections, see \cite{fraserdist, kaenmakiproj,FraserOrponen,orponenassouad}.   There has recently been intensive interest in nonlinear projections in a variety of contexts.  For example, see \cite{barany, zahl, hochmanshmerkin, pablo3}. 

For concreteness we recall the definition of the Assouad dimension, although we will not use the definition directly.   Given $F \subseteq \rd$, the Assouad dimension of $F$ is defined to be the infimum of $\alpha \geq 0$ for which there is a constant $C>0$ such that, for all $x \in F$ and scales $0<r<R$, the intersection of $F$ with the ball $B(x,R)$ may be covered by fewer than $C(R/r)^\alpha$ sets of diameter $r$.  In particular, $0 \leq \hd F \leq \ad F \leq n$.

\section{A  nonlinear projection theorem for Assouad dimension}

Our main result is a general nonlinear projection theorem for  Assouad dimension.  The nonlinear projections we consider are defined in Definition \ref{defproj}.  The definition  may seem technical, but in the applications which follow it will be obvious that these conditions are satisfied.

\begin{defn} \label{defproj}
We call $(\{\Pi_t : t \in \Omega\}, \mu, \mathbb{P})$  a \emph{generalised family of projections of $\rd$ of rank $m \geq 1$} if $\Omega$ is a metric space, $\mu$ a Borel measure on $\Omega$, $\mathbb{P}$ a Borel measure on $G(n,m)$   and:
\begin{enumerate}
\item (Domain) For all $ t \in \Omega$, $\Pi_t$ is a function  mapping $\rd$ into itself.
\item (Differentiability) For all $ z \in \rd$,   $\Pi_t$ is  a  $C^1$ map of constant rank $m$ in some  open  neighbourhood  of $z$ for $\mu$ almost all $t \in \Omega$.  That is, for $\mu$ almost all $t$,   $\Pi_t$ is  continuously differentiable in a neighbourhood of $z$ and the Jacobian $J_{z'} \Pi_t$ is a rank $m$ matrix for all $z'$ sufficiently close to $z$.

In particular, this means that for all $z \in \rd$ the map $T_z:\Omega \to G(n,m)$ given by  $T_z(t) = \textup{ker}(J_z \Pi_t)^\perp$ is  well-defined almost everywhere  (using the rank nullity theorem).
\item (Absolute continuity) For all $z \in \rd$, $\mu \circ T_z^{-1} \ll \mathbb{P}$.
\end{enumerate}
\end{defn}

 Note that the Jacobian derivatives $J_z\Pi_t$ appearing in Definition \ref{defproj} need not be  projection matrices. In most applications, for all $z$,  $\Pi_t$ will be smooth in a neighbourhood of $z$ for all but at most one point $t \in \Omega$.  One can think of the absolute continuity assumption in terms of transversality of the family $\{\Pi_t\}_t$.

\begin{thm}\label{thmmain}
Let $(\{\Pi_t : t \in \Omega\}, \mu, \mathbb{P})$ denote a generalised family of projections of $\rd$ of rank $m \geq 1$. For all non-empty bounded  $F \subseteq \rd$,
\[
\ad \Pi_t (F) \geq  \inf_{\substack{E \subseteq \rd :\\ \hd E = \ad F}}
    \underset{V \sim \mathbb{P} }{\textup{essinf}} \  \ad \pi_V (E) 
\]
for $\mu$ almost all $t \in \Omega$.
\end{thm}

 We chose to use general Borel measures $\mathbb{P}$ on $G(n,m)$ rather than the usual Grassmannian measure because this allows us to deduce dimension estimates for the exceptional set.  However, the most direct application of Theorem \ref{thmmain} is when $\mathbb{P}$ is the Grassmannian measure.

\begin{cor}\label{thmmaincor}
Let $(\{\Pi_t : t \in \Omega\}, \mu, \mathbb{P})$ denote a generalised family of projections of $\rd$ of rank $m \geq 1$, where $\mathbb{P}$ is the Grassmannian measure. For all non-empty bounded  $F \subseteq \rd$,
\[
\ad \Pi_t (F) \geq \min\{\ad F, \, m\}
\]
for $\mu$ almost all $t \in \Omega$.
\end{cor}

\begin{proof}
This follows from Theorem \ref{thmmain} and  \eqref{assouadprojection}.
\end{proof}


It is also of interest to investigate the exceptional set in Corollary \ref{thmmaincor}. Theorem \ref{thmmain} also allows one to obtain estimates on the Hausdorff dimension of the exceptional set by relating it to the Hausdorff dimension of the exceptional set in the setting of orthogonal projections.  We write $\mathcal{H}^s$ for the $s$-dimensional Hausdorff (outer) measure.  
 
\newpage

\begin{cor}\label{thmmain2}
Suppose $(\{\Pi_t : t \in \Omega\}, \mathcal{H}^s, \mathcal{H}^{u})$ is a generalised family of projections of $\rd$ of rank $m \geq 1$ for all
\[
u>\sup \hd \{ V \in G(n,m) : \ad \pi_V (E) < \lambda\}
\]
where the supremum is taken over all non-empty $E \subseteq \rd$ with $\hd E = \ad F$.  For all non-empty bounded  $F \subseteq \rd$,
\[
\ad \Pi_t (F) \geq \lambda
\]
for  all  $t \in \Omega$ outside of a set of exceptions of Hausdorff dimension at most $s$.
\end{cor}

\begin{proof}
This follows from Theorem \ref{thmmain} since,  for all $E \subseteq \rd$ with $\hd E = \ad F$, $\underset{V \sim \mathcal{H}^u}{\textup{essinf}} \  \ad \pi_V (E) \geq \lambda.$
\end{proof}

When applying Corollary \ref{thmmain2} it is useful to be able to estimate
\begin{equation} \label{eta}
\theta(s,n,m):= \sup \hd \{ V \in G(n,m) : \ad \pi_V F < \min\{\ad F, m\}\}
\end{equation}
where the supremum is taken over all sets $F \subseteq \rd$ with $\ad F = s$.  It was proved in \cite{fraserdist} that, for all integers $n>m\geq 1$ and $s \in [0,n]$,
\begin{equation} \label{etaest}
\theta(s,n,m) \leq m(n-m) - |m-s|.
\end{equation}
These bounds are simply the known (sharp) bounds for the set of exceptions to \eqref{mmproj} translated to the Assouad dimension setting \eqref{assouadprojection}.  Corollary \ref{thmmain2} is especially useful when $n=2$ and $m=1$ since Orponen's projection theorem \cite{orponenassouad} provides the sharp estimate on the Hausdorff dimension of the set of exceptions to \eqref{assouadprojection} in the planar case.

\begin{cor}\label{thmmain3}
Suppose $(\{\Pi_t : t \in \Omega\}, \mathcal{H}^s, \mathcal{H}^{u})$ is a generalised family of projections of $\rr$ of rank $1$ for all $u>0$.  For all non-empty bounded  $F \subseteq \rr$,
\[
\ad \Pi_t (F) \geq \min\{\ad F, \, 1\}
\]
for  all  $t \in \Omega$ outside of a set of exceptions of Hausdorff dimension at most $s$.
\end{cor}

\begin{proof}
This follows from Corollary \ref{thmmain2} and Orponen's projection theorem \cite[Theorem 1.1]{orponenassouad}, which shows that
\[
 \hd \{ V \in G(2,1) : \ad \pi_V (E) < \min\{\ad E, \, 1\}\} = 0
\]
for all non-empty $E \subseteq \rr$. In particular, $\theta(s,2,1) = 0$ for all $s \in [0,2]$.
\end{proof}

In certain cases one may only be interested in projections of sets $F$ contained in a subset  $U \subseteq \rd$.  In this case the results in this section can be applied under the weaker assumption that the domain of each $\Pi_t$  is an open set $U_0 \supseteq U$, and the differentiability and absolute continuity assumptions hold only  for all $z \in U_0$.  This version of the theorem can be deduced directly  from Theorem \ref{thmmain} appealing to the Whitney extension theorem.  We omit the details.

\section{Applications}

\subsection{Distance sets}

The \emph{distance set problem}, originating with the paper \cite{distancesets}, is a well-studied problem in geometric measure theory.  It was received a lot of attention in the literature in the last few years, see for example \cite{fraserdist, guth,keletishmerkin,orponenAD,pablo,pablo2,pablo3}.   Given $F \subseteq \rd$, the \emph{distance set} of $F$ is
\[
D(F) = \{|x-y| : x,y \in F\} \subseteq [0,\infty).
\]
The distance set problem is to understand the relationship between the dimensions of $F$ and $D(F)$. It is conjectured that if $F \subseteq \rd$ is Borel and $\hd F \geq n/2$, then $\hd D(F) = 1$.  This conjecture  is open for all $n \geq 2$.  The same conjecture can also be made with Hausdorff dimension replaced by Assouad dimension.  This conjecture is also open, although it was  proved in \cite{fraserdist} that for $F \subseteq \rr$, $\ad F > 1$ guarantees $\ad D(F) = 1$.  We are able to fully resolve the Assouad dimension version of the distance set problem in the plane, both dealing with the awkward `critical case' $\ad F = 1$ and    providing sharp estimates for sets with Assouad dimension less than 1.   We emphasise that we do not required $F$ to be bounded or Borel.

\begin{thm}\label{distmain0}
For all non-empty  sets $F \subseteq \rr$,
\[
\ad D(F) \geq \min\{\ad F, \, 1\}.
\]
\end{thm}

Theorem \ref{distmain0} follows immediately from the more general Theorem \ref{distmain} below.  Theorem \ref{distmain0} is sharp, as the following corollary shows. For comparison, it was already observed in \cite{fraserdist} that, for all $s \in [0,2]$, $\sup\{ \ad D(F) : F \subseteq \rr \text{ and } \ad F \leq s\} = 1.$
\begin{cor}
For all $s \in [0,1]$,
\[
\inf\{ \ad D(F) : F \subseteq \rr \text{ and } \ad F \geq s\} = s.
\]
\end{cor}

\begin{proof}
The lower bound ($\geq s$) follows from Theorem \ref{distmain0}.  The upper bound ($\leq s$) follows by a standard construction: see,  for example, \cite[Section 3.3.1]{fraserdist}. Briefly, for $s \in (0,1)$, let $F\subseteq [0,1]$ be a self-similar set generated by $\lceil N^s \rceil$ equally spaced homotheties with contraction ratio $1/N$.  This ensures that $\ad F \geq s$.  Moreover, for  $V=\textup{span}(1,-1) \in G(2,1)$,   the distance set $D(F)$ has Assouad dimension no more than that of  $\pi_V(F \times F)$, which is itself a self-similar set generated by $2\lceil N^s \rceil - 1$ equally spaced homotheties with contraction ratio $1/N$.  As  $N \to \infty$, $\ad D(F)$ approaches $s$.
\end{proof}

The next theorem   considers the   distance problem in $\rd$ for arbitrary $n \geq 2$.  It shows that the set of exceptions to \eqref{assouadprojection} plays a  role.   Consider projections of sets of Assouad dimension $s$ from $\rd$ onto $m$-dimensional subspaces and let $\theta(s,n,m)$ be the largest possible Hausdorff dimension of  set of exceptions to \eqref{assouadprojection}, recall \eqref{eta}.  

\begin{thm}\label{distmain}
For all non-empty  sets $F \subseteq \rd$,
\[
\ad D(F) \geq \min\{\ad F-\theta, \, 1\}
\]
where $\theta = \theta(\ad F, n, 1)$.
\end{thm}

The proof of Theorem \ref{distmain} requires some technical machinery we have not yet introduced.  Therefore we delay the proof to Section \ref{distproof}.  Theorem \ref{distmain0} follows from Theorem \ref{distmain} together with Orponen's projection theorem \cite[Theorem 1.1]{orponenassouad} which states that $\theta(s,2,1)=0$ for all $s \in [0,2]$.   Given this connection with the exceptional set, it is natural to ask when information on the exceptional set solves the distance problem in higher dimensions.  Applying \eqref{etaest}, we get
\[
\theta(s,n,1) \leq \min\{n-s, n+s-2\}.
\]
Combining this with Theorem \ref{distmain} we get the following, which does \emph{not} improve over known results, e.g. \cite[Theorem 2.5]{fraserdist}, but provides  a somewhat different proof.
\begin{cor}
If  $F \subseteq \rd$ with $\ad F \geq (n+1)/2$, then $\ad D(F) =1$.
\end{cor}
The bound \eqref{etaest}  for $\theta(s,n,m)$ was proved by applying the bounds for the exceptional set in the Marstrand-Mattila projection theorem \eqref{mmproj}.  Orponen's projection theorem is reason to believe that much better bounds are available in the Assouad dimension case. Indeed, if we could prove that $\theta(n/2,n,1) \leq n/2-1$, then  all $F \subseteq \rd$ with $\ad F \geq n/2$ would satisfy $\ad D(F) =1$.  However, this is not true, at least for $n=3$.

\begin{prop}
For all $n \geq 2$, $\theta(s,n,1) \geq n-2$ for all $s \in[1,2)$.
\end{prop}

\begin{proof}
Let $V_0 \in G(n,n-1)$ and $E \subseteq V_0$ be contained in a line segment with   $\ad E = s-1$.  Let $F=E \times [0,1] \subseteq \rd$.  Clearly $\ad F = s$ and for all $V \in G(n,1)$ with $V \subseteq V_0$ the projection  $\pi_V(F)$ is the image of $E$ under a similarity (possible with contraction ratio 0).  Therefore, for all such $V$,
\[
\ad \pi_V(F) = s-1<1 = \min\{s,1\}.
\]
The Hausdorff dimension of the set of such $V$ is the same as that of $G(n-1,1)$ which is $n-2$.
\end{proof}

\subsubsection{Pinned distance sets}

A related problem is to consider pinned distance sets.  Given $x \in \rd$, the \emph{pinned distance set} of $F \subseteq \rd$ at $x$ is
\[
D_x(F) = \{|x-y| : y \in F\}.
\]
If $x \in F$, then $D_x(F) \subseteq D(F)$. Here the conjecture is that if $F$ is Borel and $\hd F \geq d/2$, then there should exist a pin $x \in F$ such that $\hd D_x(F) = 1$ (or even \emph{many} pins).   We are also able to prove some results on pinned distance sets in the Assouad dimension setting.

\begin{thm}\label{distpind}
Let $F \subseteq \rd$ be a non-empty bounded  set.  For Lebesgue almost all $x \in \rd$,
\[
\ad D_x( F) \geq  \min\{\ad F, \, 1\}.
\]
Moreover, the set of exceptional $x$ where this does not hold has Hausdorff dimension at most $1+\theta(\ad E, n, 1)  \leq  n-|\ad F - 1|$.
\end{thm}

\begin{proof}
For $t \in \mathbb{R}^n$, consider the maps $\Pi_t : \rd \to \mathbb{R}$ defined by
\[
\Pi_t(x) = |x-t|.
\]
Then,
\[
T_z(t) = \textup{ker}(J_z \Pi_t)^\perp = \textup{span}(z-t) \in G(n, 1)
\]
is defined for all $t \neq z$.  Since the preimage of $V \in G(n,1)$ under $T_z$ is a line (with Hausdorff dimension   1), the triple $(\{\Pi_t : t \in \rd\},  \mathcal{H}^{u+1}, \mathcal{H}^u)$ is   a generalised family of projections of $\rd$ of rank $1$ for all $u>0$.  The results follow by applying Corollary \ref{thmmaincor} (with $u=n-1$) and Corollary \ref{thmmain2} (with $u>\theta(\ad F, n, 1)$, recalling \eqref{eta})  observing that $\Pi_t(F) = D_t(F)$.  The quantitative bound comes from \eqref{etaest}.
\end{proof}

We can upgrade this result in the planar case since $\theta(s,2,1)=0$ for all $s \in [0,2]$, which was proved in  \cite{orponenassouad}.

\begin{cor}\label{distpin}
Let $F \subseteq \rr$ be a non-empty bounded set.  For all $x \in \rr$   outside of a set of exceptions of Hausdorff dimension at most $1$,
\[
\ad D_x( F) \geq  \min\{\ad F, \, 1\}.
\]
Therefore, if $\hd F>1$, then there exists $x \in F$ such that $\ad D_x( F) =1$.
\end{cor}

Shmerkin \cite{pablo2} proved that if $F \subseteq \rr$ is a Borel set with equal Hausdorff and packing dimension strictly larger than 1, then  there exists $x \in F$ such that $\hd D_x( F) =1$.

\subsubsection{Distance sets with respect to other norms} \label{norm}

It is also natural to consider the distance set (and pinned distance set) problem with respect to norms other than the Euclidean norm.  That is, given a norm $\| \cdot \|$ on $\rd$, the distance set of $F \subseteq \rd$ with respect to $\| \cdot \|$ is
\[
D^{\| \cdot \|}(F) =  \{\|x-y\| : x, y \in F\}
\]
with the obvious analogous definition of pinned distance sets $D_x^{\| \cdot \|}(F) $.  Whether or not we expect the same results to hold turns out to depend on the curvature of the unit ball in the given norm.   Theorems \ref{distmain0}, \ref{distmain}, \ref{distpind} and Corollary \ref{distpin} hold in this more general setting provided  the  boundary of the unit ball  $\partial B$ is  a $C^1$ manifold and the associated Gauss map cannot decrease Hausdorff dimension (that is, $\hd g(E) \geq \hd E$ for all $E \subseteq \partial B$, where $g: \partial B \to S^{n-1}$ is the Gauss map).  For example, this holds if $\partial B$  is  a $C^2$ manifold with non-vanishing Gaussian curvature, since in that case the Gauss map is a diffeomorphism, see  \cite[Corollary 3.1]{topology}.

  Let  $\Pi_t^{\| \cdot \|}: \rd \to \mathbb{R}$ denote the pinned distance map with respect to a general norm, that is, $\Pi^{\| \cdot \|}_t(x) = \|x-t\|$, and let $T_z(t) = \textup{ker}(J_z \Pi_t^{\| \cdot \|})^\perp$.   If the boundary of the unit ball $\partial B$ is  $C^1$, then the restriction of $T_z$ to $(\partial B+z)$ coincides with the Gauss map (identifying antipodal points in $S^{n-1}$ and then identifying with $G(n,1)$).  Therefore, provided the Gauss map cannot decrease Hausdorff dimension,
\[
\mathcal{H}^{u+1} \circ T_z^{-1} \ll \mathcal{H}^u
\]
for all $u>0$.  This observation allows the proof of  Theorem \ref{distpind} (and Corollary \ref{distpin})  to  go through in this more general setting.  The proof of Theorem  \ref{distmain} (and Theorem \ref{distmain0}) is deferred until Section \ref{distproof} and so we also defer discussion of its extension to general norms.

The assumption of  non-vanishing Gaussian curvature is natural when studying distance sets.  Indeed, for certain ``flat norms'' the analogous results do not hold, see \cite{poly}.  See recent examples \cite{guth, pablo3} where results are obtained for the Hausdorff dimension of distance sets under the assumption that the unit ball is $C^\infty$ and $C^2$, respectively, in addition to having non-vanishing Gaussian curvature.   It is perhaps noteworthy that we only require $C^1$  regularity and a weaker condition on the Gauss map.  For example,    our techniques  allow for   the Gaussian curvature to vanish on a countable set of points.

\subsection{A radial projection theorem for Assouad dimension}

Radial projections are perhaps the most natural family of projections alongside orthogonal projections.  Given $t \in \rd$, the radial projection $\pi_t$ maps $\rd\setminus \{t\}$ onto the boundary of the sphere centred at $t$ with radius 1.  Specifically, $\pi_t(x) \in t+S^{n-1}$ is defined by
\[
\pi_t(x) = \frac{x-t}{|x-t|} +t
\]
and we define $\pi_t(t)=t$ for convenience.  Radial analogues of results such as the Marstrand-Mattila projection theorem are known and turn out to be important in their own right in a variety of settings.  For example, Orponen's radial projection theorem \cite{orponenradial} has proved a useful tool in in studying the distance set problem, see \cite{guth, keletishmerkin}.  Recall the definition of $\theta$ from \eqref{eta}.

\begin{thm}\label{radialthm1}
Let $F \subseteq \rd$ be a non-empty bounded  set.  For Lebesgue almost all $t \in \rd$,
\[
\ad \pi_t( F) \geq  \min\{\ad F, \, n-1\}.
\]
Moreover, the set of exceptional $t \in \rd$ where this does not hold has Hausdorff dimension at most $1+\theta(\ad F, n, n-1) \leq \min\{\ad F+1, 2n-1-\ad F \}$.
\end{thm}

\begin{proof}
For all $z \in \rd$ and $t \neq z$, $\pi_t$ is smooth on $B(z,|z-t|/2)$ and 
\[
T_z(t) = \textup{ker}(J_z \pi_t)^\perp = \textup{span}(z-t)^\perp \in G(n, n-1).
\]
Since the preimage of $V \in G(n,n-1)$ under $T_z$ is again a line, the triple  $(\{\pi_t: t \in \rd\}, \mathcal{H}^{u+1}, \mathcal{H}^u)$ is  a generalised family of projections of $\rd$ of rank $n-1$.  The results follow by applying Corollary \ref{thmmaincor} (with $u=n-1$) and Corollary \ref{thmmain2} (with $u>\theta(\ad F, n, n-1)$, recalling \eqref{eta}).  The quantitative bound comes from \eqref{etaest}.
\end{proof}

We note that $S^{n-1}$ can be replaced by any smooth enough  $(n-1)$-dimensional `radially accessible' set.  More precisely, let $\mathcal{S} \subseteq \rd$ be a simply connected compact $(n-1)$-dimensional  $C^1$ manifold, with the property that for all $x \in \rd \setminus\{0\}$ the intersection
\[
 \{\lambda x : \lambda>0\} \cap \mathcal{S}
\]
is a singleton, which we denote by $\mathcal{S}(x)$.  Then the family of radial projections onto $\mathcal{S}$ with centre $t \in \rd$ given by
\[
\pi^\mathcal{S}_t(x) = \mathcal{S}(x-t)+t
\]
also satisfies the conclusion of Theorem \ref{radialthm1}.  Moreover, the exceptional set does  not  depend on $\mathcal{S}$ and so the conclusion holds for all $\mathcal{S}$ simultaneously.

We   obtain a sharp result concerning the dimension of the exceptional set in Theorem \ref{radialthm1} in the planar case  since $\theta(s,2,1)=0$ for all $s \in [0,2]$ by Orponen's projection theorem \cite{orponenassouad}.

\begin{cor}\label{radialthm}
Let $F \subseteq \rr$ be a non-empty bounded set.  Then
\[
\ad \pi_x( F) \geq  \min\{\ad F, \, 1\}
\]
for all  $x \in \rr$   outside of a set of exceptions of Hausdorff dimension at  most 1.
\end{cor}

Corollary \ref{radialthm}  is clearly sharp since a line segment will radially project to a single point for all $t$ in the affine span of the line segment. 

\subsection{A sum-product theorem}

`Sum-product' results in additive combinatorics refer to a wide range of phenomena regarding the `independence' of multiplication and addition.  For example,  for a  set $F \subset (0,1)$, one cannot expect the product set
\[
FF = \{ xy : x, y \in F\}
\]
and the sumset
\[
F+F = \{ x+y : x,y \in F\}
\]
to be small simultaneously.     If $F$ is finite, then size means cardinality and this statement is made precise by the Erd\H{o}s-Szemer\'{e}di theorem.   If $F$ is infinite then it is natural to describe size in terms of  dimension.   The following is a sum-product type result for Assouad dimension, where we are also able to consider independence of other operations such as addition and exponentiation.  For $F \subseteq (0, \infty)$, we write
\[
F^F = \{x^y : x, y \in F\}.
\]

\begin{thm}\label{sumproductthm}
Let $F \subseteq \mathbb{R}$ be a non-empty bounded set with $\hd F >0$.  Then
\[
\ad (FF+F) \geq  \min\{2\ad F, \, 1\},
\]
and, if $F \subseteq (0,\infty)$,
\[
\ad (F^F+F) \geq  \min\{2\ad F, \, 1\}.
\]
\end{thm}

\begin{proof}
For $t \in \mathbb{R}$, consider the  family of projections $\Pi_t : \rr \to \mathbb{R}$ defined by
\[
\Pi_t(x,y) = tx+y.
\]
Applying Corollary  \ref{thmmain3} with $s=u$ to the cartesian product $F \times F = \{(x,y) : x,y \in F\}$ (not to be confused with $FF$) we get
\[
\hd \{ t : \ad \Pi_t(F \times F)< \min\{\ad (F \times F), \, 1\} \} = 0.
\]
Since $\hd F>0$, there must exist $t \in F$ such that 
\[
\ad \Pi_t(F \times F) \geq  \min\{\ad (F \times F), \, 1\} = \min\{2\ad F, \, 1\}.
\]
The result follows since $\Pi_t(F \times F) = tF+F \subseteq FF+ F$. The fact that $\ad (F \times F) =  2\ad F$ can be found in, for example, \cite[Theorem A.5 (5)]{luk}.  The second result is proved similarly, but  the details are more involved. For $t > 0$, consider the  family of projections $\Pi_t : \rr \to \mathbb{R}$ defined by   
\[
\Pi_t(x,y) = t^x+y.
\]
Here
\[
T_{(x,y)}(t) = \textup{ker}(J_{(x,y)} \Pi_t)^\perp = \textup{span}\Big(1,\frac{t^{-x}}{\log(t)}\Big) \in G(2, 1)
\]
is defined for all $t > 0$ (with the obvious interpretation  $\textup{span}(1, -\infty)=\textup{span}(0,1)$ when $t=1$).  Although $T_{(x,y)}:(0,\infty) \to G(2,1)$  is not generally surjective or injective, we still have
\[
\mathcal{H}^s \circ T_{(x,y)}^{-1} \ll \mathcal{H}^s
\]
for all $s>0$.  Therefore, by applying Corollary \ref{thmmain3} to $F \times F$,
\[
\hd \{ t > 0 : \ad \Pi_t(F \times F)< \min\{\ad (F \times F), \, 1\} \} = 0.
\]
Since $\hd F>0$ and $F \subseteq [0,\infty)$, there must exist $t \in F$ such that 
\[
\ad \Pi_t(F \times F) \geq  \min\{\ad (F \times F), \, 1\} = \min\{2\ad F, \, 1\}.
\]
The result follows since $\Pi_t(F \times F) = t^F+F \subseteq F^F+ F$.
\end{proof}

This example was partly motivated by Orponen's paper \cite{orponenAD}.  Orponen \cite[Corollary 1.5]{orponenAD} proved that if $F \subseteq \mathbb{R}$ is compact, Ahlfors-David regular, and has $\hd F \geq 1/2$, then
\[
\pd (FF+FF-FF-FF) = 1,
\]
where $\pd$ denotes packing dimension.  We are able to provide a much stronger result, but with packing dimension replaced by Assouad dimension.  Notably, the set $F$ need not be Ahlfors-David regular, we consider the much smaller set $FF+F$, and we obtain estimates for sets with arbitrarily small dimension.  We note that since the family of projections used to handle $FF+F$ in Theorem \ref{sumproductthm} are orthogonal, this result could be deduced directly from  Orponen's projection theorem.  The set $F^F+F$ requires our nonlinear theorem, however.  Finally, we observe that many other sets constructed from $F$ can be handled in this way --- or even sets constructed from a collection of sets, rather than the single set $F$.  We leave the details to the interested reader.  

\subsection{Dimension of sumsets}

As a final application we revisit one of the situations where Peres and Schlag \cite{peresschlag} were able to apply their nonlinear projection theorem.  Given two non-empty sets $E,F \subseteq \mathbb{R}$ with sufficient `arithmetic independence', one might hope for $\dim (E+F) = \min\{ \dim E + \dim F, \, 1\}$.  This can fail for many reasons but if we parameterise $F$ in a transversal enough way, then we can recover this formula generically.  Following \cite{peresschlag}, for $\lambda \in (0,1/2)$ we let
\[
F_\lambda = \bigg\{ \sum_{n\geq 1} i_n\lambda^n :  i_n \in \{0,1\}\bigg\}
\]
and consider $E+F_\lambda$ for generic $\lambda$.  For all $\lambda \in (0,1/2)$, $F_\lambda$ is a compact self-similar Cantor set with $\hd F_\lambda = \ad F_\lambda = -\log 2/\log\lambda $. The following result also holds for more general homogeneous Cantor sets, but we omit the details.

\begin{thm} \label{sumset}
Let $E \subseteq \mathbb{R}$ be non-empty.  Then, for almost all $\lambda \in (0,1/2)$,
\[
\ad (E+F_\lambda) \geq   \min\{\ad E+ \ad F_\lambda, \, 1\}.
\]
Moreover, the set of exceptional $\lambda$ in a given interval $(a,b)  \subseteq (0,1/2)$ for which this does not hold has Hausdorff dimension at most $\ad E+\ad F_b $. 
\end{thm}

One of the distinguishing features of this result is that the generic dimension bound depends on the parameter $\lambda$.   The proof will be a straightforward combination of our approach and the result of Peres and Schlag.  Nevertheless, we delay the proof until Section \ref{sumsetproof}.






\section{Proofs of nonlinear projection theorems}

\subsection{Tangents}

The tangent structure of a set is intimately  related to the Assouad dimension and it is via the tangent structure that we will prove Theorem \ref{thmmain}.  Mackay and Tyson \cite{mackaytyson} pioneered the theory of \emph{weak tangents} in the context of Assouad dimension.  Weak tangents are limits of sequences of blow-ups of a given set with respect to the Hausdorff metric.  Rather than use weak tangents directly, it is more convenient for us to use the non-symmetric Hausdorff distance defined by
\[
\rho_\mathcal{H}(A,B) = \sup_{a \in A} \inf_{b \in B} |a-b|
\]
for non-empty closed sets $A,B \subseteq \rd$.  The Hausdorff metric  is then defined as 
\[
d_\mathcal{H}(A,B) = \max\{\rho_\mathcal{H}(A,B), \rho_\mathcal{H}(B,A)\}
\]
for non-empty compact sets $A,B \subseteq \rd$.  In what follows we choose to approximate using $\rho_\mathcal{H}$ rather than $d_\mathcal{H}$.  An alternative would have been to approximate using $d_\mathcal{H}$ via \emph{subsets}, but we found this more cumbersome. This approach was used, for example, in \cite[Definition 3.6]{fraserrobinson} with the terminology \emph{weak pseudo tangent}.  Another minor variation we make on the usual theory of weak tangents is to allow some flexibility in the blow-ups: they need not be via strict similarities.    This approach was used, for example, in \cite[Proposition 7.7]{fraserassouad} with the terminology \emph{very weak tangents}. To simplify exposition and terminology, we simply refer to \emph{tangents}. We write $B(x,r)$ for the closed ball centred at $x \in \rd$ with radius $r>0$.

\begin{defn}\label{weaktangentdef}  
Let $E, F \subseteq \mathbb{R}^n$ be closed sets with $E \subseteq B(0,1)$.  Suppose there exists a sequence of maps $S_k: \mathbb{R}^n \to \mathbb{R}^n$ and  constants  $a_k, b_k >0$ with $\sup_k (b_k/a_k) <\infty$ such that
\[
a_k|x-y| \leq |S_k(x)-S_k(y)|  \leq b_k|x-y|
\]
for all $x,y \in S_k^{-1}(B(0,1))$ and suppose that
\[
\rho_\mathcal{H}(E, S_k(F)) \to 0
\]
as $k \to \infty$.  Then we call $E$ a   \emph{tangent} to $F$.  If each $S_k$ is a homothety, that is, $S_k(x) = c_k x + t_k$ for some $c_k>0$ and $t_k \in \rd$, and  $c_k \to \infty$, then we call $E$ a \emph{simple   tangent} to $F$.
\end{defn}

The maps $S_k$ in Definition \ref{weaktangentdef} blow-up the set $F$ around $z_k = S_k^{-1}(0)$.  If the limit $z = \lim_{k \to \infty} z_k \in \rd$ exists, then we call $z$ the \emph{focal point} of $E$.  Note that if $F$ is compact and $E$ is a simple  tangent to $F$, then we may assume (by taking a subsequence if necessary) that the focal point exists and, moreover, is a point in $F$. The following is a minor variant on a result of Mackay and Tyson \cite[Proposition 6.1.5]{mackaytyson}.

\begin{thm}\label{weaktangent}
Let $F \subseteq \mathbb{R}^d$ be closed and $E \subseteq \mathbb{R}^d$ be  a   tangent to $F$.   Then $\dim_\mathrm{A} F \geq \dim_\mathrm{A} E$.
\end{thm}

The following result of K\"aenm\"aki, Ojala and Rossi \cite[Proposition 5.7]{anti3} shows that Theorem \ref{weaktangent} has a useful converse.

\begin{thm} \label{goodweaktangent}
Let $F \subseteq \mathbb{R}^d$ be closed and non-empty.  Then there exists a compact set  $E \subseteq \mathbb{R}^d$ with $\hd E = \ad F$ such that $E$  is a simple   tangent to $F$. 
\end{thm}

\subsection{Orthogonal projections of tangents are tangents of nonlinear projections}

The key technical result required to prove Theorem \ref{thmmain} is the following proposition.  It states that there is an appropriately chosen orthogonal projection of a simple   tangent, which is a   tangent to a given nonlinear projection.

\begin{prop} \label{keylemma}
Let $F \subseteq \rd$ be non-empty and compact.  Suppose $E$ is a simple  tangent to $F$  with focal point $z \in F$.  Further suppose that $t \in \Omega$ is such that $\Pi_t$ is $C^1$ and of constant rank $m \geq 1$ in a neighbourhood of $z$.  Then  $\pi_V(E)$ is a   tangent to $\Pi_t(F)$ for $V=\textup{ker}(J_z \Pi_t)^\perp \in G(n,m)$.
\end{prop}

Before proving Proposition \ref{keylemma}, we provide some preliminary results. We may assume for convenience that $E \subseteq B(0,1/2)$.  Let $S_k$ be a sequence of homothetic similarities of $\rd$ such that
\[
\rho_\mathcal{H}(E, S_k(F) ) \to 0.
\]
  Write $c_k>0$ for the similarity ratio of $S_k$ and $t_k \in \rd$ for the associated translation.  Let $z_k \in F$ be such that $S_k(B(z_k,c_k^{-1})) = B(0,1)$ and $z = \lim_{k \to \infty} z_k \in F$ be the focal point of $E$.  (Note that $0 = S_k(z_k) = c_kz_k+t_k$.)   Let $V=\textup{ker}(J_z \Pi_t)^\perp$ and $V_k=\textup{ker}(J_{z_k} \Pi_t)^\perp$, noting that $V_k, V \in G(n,m)$ for large enough $k$ by the differentiability assumption.  Moreover, $V_k \to V$ in the \emph{Grassmannian metric} $d_G$, defined by
\[
d_G(U,U') = d_{\mathcal{H}} \Big(U \cap B(0,1),\,  U' \cap B(0,1) \Big)
\]
for $U,U' \in G(n,m)$.  This convergence is guaranteed by the assumption that $\Pi_t $ is continuously differentiable  in a neighbourhood of $z$, and therefore $\textup{ker}(J_{z'} \Pi_t)^\perp$ varies continuously for  $z'$ sufficiently close to $z$.

  There exists a  constant $c=c(z,t)>0$ such that, for all   $k$ sufficiently large and  all $x,y \in \textup{ker}(J_{z_k} \Pi_t)^\perp$,
\begin{equation} \label{eccentric}
|(J_{z_k} \Pi_t)(x) - (J_{z_k} \Pi_t)(y)| \geq c \|J_{z_k} \Pi_t \| |x-y|,
\end{equation}
where $\| \cdot \|$ denotes the operator norm. This can be guaranteed since $z_k \to z$, $\Pi_t$ is  continuously differentiable in a neighbourhood of $z$, and $J_{z} \Pi_t$ is injective and linear on $\textup{ker}(J_{z} \Pi_t)^\perp$. 

Fix  $\eps \in (0,1/10)$ satisfying 
\begin{equation} \label{epssmall}
0<\eps<(c/8) \|J_z\Pi_t\|
\end{equation}
where $c=c(z,t)>0$ is the constant from \eqref{eccentric}. This is not an issue since $z$ and $t$ are fixed.

Define $U_k: \Pi_t(B(z_k,c_k^{-1})) \to \rd$ by  $U_k = S_k \circ U_k^0$ where $U_k^0:\Pi_t(B(z_k,c_k^{-1}))  \to B(z_k, 2c_k^{-1})$ is defined by letting $U_k^0(u)$  be the unique point in the intersection
\[
\Pi^{-1}_t(u) \cap (V_k+z_k) \cap B(z_k,2c_k^{-1}).
\]
\begin{lma}
The map $U_k^0$ is well-defined for sufficiently large $k$.
\end{lma}

\begin{proof}
Throughout this proof we restrict $\Pi_t$ to a neighbourhood of $z$ such that it is $C^1$ and of constant rank $m$.  By the implicit function theorem, the level set $\Pi^{-1}_t(u)$ is a simply connected $(n-m)$-dimensional  $C^1$ manifold  which intersects  $B(z_k,c_k^{-1})$ since $u \in \Pi_t(B(z_k,c_k^{-1}))$. This  follows by expressing the action of $\Pi_t$ near $z$ in local coordinates.   Moreover, since $\Pi_t$ is differentiable,   vectors $v$ in the tangent space $T_x \Pi^{-1}_t(u)$ at  $x \in \Pi^{-1}_t(u)$ coincide with directional  derivatives of $\Pi_t$ at $x$ in direction $v$.   For the manifold $\Pi^{-1}_t(u) $ to intersect $(V_k+z_k) $ more than once, or not at all, inside $B(z_k,2c_k^{-1})$ we would require the tangent spaces  of $\Pi^{-1}_t(u)$ at points inside $B(z_k,2c_k^{-1})$ to differ from $\textup{ker}(J_{z_k}\Pi_t)$ by more than 1/100 (in the Grassmannian metric, say).  This is  impossible for large enough $k$ since  $\Pi_t$  is continuously differentiable in a neighbourhood of $z$.
\end{proof}

We will use the maps $U_k$ to show that $\pi_V(E) \subseteq B(0,1) \cap V$ is a   tangent to $\Pi_t(F)$.  Therefore we must show these maps satisfy the conditions from Definition \ref{weaktangentdef}.  Since $S_k$ is a homothety, it is sufficient to demonstrate that $U_k^0$ satisfies the conditions.  This is the content of the next lemma. Note that we only need to consider points which map into $B(0,1)$ under $U_k$, which is consistent with the domain of $U_k^0$ being $\Pi_t(B(z_k,c_k^{-1}))$.  We may extend $U_k^0$ (and thus $U_k$) to a mapping on the whole of $\rd$ if we wish, but this is not really necessary.

\begin{lma} \label{uk}
For sufficiently large $k$, for all $x,y \in \Pi_t S_k^{-1}(B(0,1)) = \Pi_t(B(z_k,c_k^{-1}))$
\[
\frac{1}{(2+c)\| J_{z}\Pi_t\|}  |x-y|  \leq |U_k^0(x)-U_k^0(y)|  \leq \frac{8}{c\| J_{z}\Pi_t\|}  |x-y| 
\]
where $c$ is the constant from \eqref{eccentric}.
\end{lma}

\begin{proof}
Since $\Pi_t$ is continuously differentiable  in a neighbourhood of $z$, we may assume $k$ is large enough to ensure 
\begin{equation} \label{totalderivative} 
|\Pi_t(b) - \Pi_t(a) - (J_{a}\Pi_t) (b-a)| \leq \eps|b-a|
\end{equation}
for all $a,b \in B(z_k,c_k^{-1})$.  We may also assume $k$ is large enough to ensure 
\begin{equation} \label{unidiff}
(1/2)\| J_z\Pi_t \|  \leq \| J_a\Pi_t \| \leq 2\| J_z\Pi_t \| 
\end{equation}
for all $a \in B(z_k,c_k^{-1})$.  This  estimate can be achieved because $ J_z\Pi_t $ is  continuous at $z$ and $\|J_z\Pi_t\| >0$.  These facts are guaranteed since  $\Pi_t$ is continuously differentiable  in a neighbourhood of $z$ and $J_z\Pi_t$ has strictly positive rank, respectively.   Finally, we may assume $k$ is large enough to guarantee
\begin{equation} \label{bound1}
1/2 \leq  \frac{ |(J_{z_k}\Pi_t) (x-y)|}{ |(J_{x}\Pi_t) (x-y )|} \leq 2
\end{equation}
for all $x,y \in B(z_k, c_k^{-1}) \cap (V_k+z_k)$.  This can be achieved since $J_{z_k}\Pi_t$ is linear and  injective on $V_k$ and $J_{x}\Pi_t$ continuous in $x$ in a neighbourhood of $z$. In particular, $J_{z_k}\Pi_t \to J_{z}\Pi_t$.

 Fix distinct $x,y \in \Pi_t(B(z_k,c_k^{-1}))$.  Since $U_k^0(x) - U_k^0(y) \in V_k$, by  \eqref{eccentric} and \eqref{unidiff},
\begin{eqnarray}
 |(J_{z_k}\Pi_t) (U_k^0(x) - U_k^0(y) )| &  \geq &  c \| J_{z_k}\Pi_t\|  |U_k^0(x) - U_k^0(y) | \nonumber  \\ \nonumber \\
& \geq &  (c/2) \| J_{z}\Pi_t\|  |U_k^0(x) - U_k^0(y) |. \label{chain1}
\end{eqnarray}
Moreover, using the fact that $U_k^0$ is injective,
\begin{eqnarray}
&\,&  \hspace{-18mm}  |(J_{z_k}\Pi_t) (U_k^0(x) - U_k^0(y) )| \nonumber  \\ \nonumber \\
& = &  \frac{ |(J_{z_k}\Pi_t) (U_k^0(x) - U_k^0(y) )|}{ |(J_{U_k^0(x)}\Pi_t) (U_k^0(x) - U_k^0(y) )|} |(J_{U_k^0(x)}\Pi_t) (U_k^0(x) - U_k^0(y) )| \nonumber  \\ \nonumber \\
& \leq & 2|\Pi_tU_k^0(x)-\Pi_tU_k^0(y) | + 2\eps |U_k^0(x)-U_k^0(y)|  \qquad \text{by \eqref{totalderivative} and \eqref{bound1}} \nonumber \\ \nonumber  \\
&= & 2|x-y| + 2\eps |U_k^0(x)-U_k^0(y)|  \label{chain2}
\end{eqnarray}
since $\Pi_tU_k^0$ is the identity on $\Pi_t(B(z_k,c_k^{-1}))$. Combining \eqref{chain1} and \eqref{chain2} and using \eqref{epssmall} yields
\[
|U_k^0(x) - U_k^0(y) |  \leq \frac{2}{(c/2)\| J_{z}\Pi_t\| -2\eps} |x-y| \leq  \frac{8}{c\| J_{z}\Pi_t\|} |x-y|
\]
as required.  The lower bound is similar.  By  the  definition of the operator norm $\| \cdot \|$ and  \eqref{unidiff},
\begin{eqnarray}
 |(J_{z_k}\Pi_t) (U_k^0(x) - U_k^0(y)) | & \leq &   \| J_{z_k}\Pi_t\|   |U_k^0(x) - U_k^0(y) | \nonumber \\ \nonumber \\
 & \leq &   2\| J_{z}\Pi_t\|   |U_k^0(x) - U_k^0(y) |.   \label{chain3}
\end{eqnarray}
Moreover, using the fact that $U_k^0$ is injective,
\begin{eqnarray}
 &\,&  \hspace{-18mm} |(J_{z_k}\Pi_t) (U_k^0(x) - U_k^0(y) )|  \nonumber  \\ \nonumber \\
& = &  \frac{ |(J_{z_k}\Pi_t) (U_k^0(x) - U_k^0(y) )|}{ |(J_{U_k^0(x)}\Pi_t) (U_k^0(x) - U_k^0(y) )|} |(J_{U_k^0(x)}\Pi_t) (U_k^0(x) - U_k^0(y) )| \nonumber  \\ \nonumber \\
& \geq & (1/2)|\Pi_tU_k^0(x)-\Pi_tU_k^0(y) | -(\eps/2) |U_k^0(x)-U_k^0(y)|  \qquad \text{by \eqref{totalderivative} and \eqref{bound1}} \nonumber \\ \nonumber  \\
&= & (1/2)|x-y | - (\eps/2) |U_k^0(x)-U_k^0(y)|  \label{chain4}
\end{eqnarray}
since $\Pi_tU_k^0$ is the identity on $\Pi_t(B(z_k,c_k^{-1}))$. Combining \eqref{chain3} and \eqref{chain4} yields
\[
|U_k^0(x) - U_k^0(y) |  \geq \frac{1/2}{\| J_{z}\Pi_t\| +\eps/2} |x-y| \geq  \frac{1}{(2+c)\| J_{z}\Pi_t\|} |x-y|
\]
as required.  
\end{proof}

The next result is a technical approximation which says that close to $z_k$ the composition $U_k^0 \Pi_t$ behaves very much like orthogonal projection onto $V+z_k$.

\begin{lma} \label{lma1}
For sufficiently large $k \geq 1$,
\[
\sup_{w \in B(z_k, c_k^{-1})} |S^{-1}_k \pi_V  S_k (w) - U_k^0 \Pi_t(w)| \leq 2 c_k^{-1} \eps.
\]
\end{lma}

\begin{proof}
Let $w \in B(z_k, c_k^{-1})$ and write $u = \Pi_t(w)$.  Then $U_k^0 \Pi_t(w) = \Pi_t^{-1}(u) \cap (V_k+z_k) \cap B(z_k,2c_k^{-1})$.  For sufficiently large $k$, the tangent spaces of  the manifold $\Pi_t^{-1}(u)$ are in an $\eps$-neighbourhood of   $\textup{ker}(J_{z_k}\Pi_t) = V_k^\perp$ (in the Grassmannian metric) and since $|w-z_k| \leq c_k^{-1}$ we conclude that
\[
|U_k^0 \Pi_t(w) - \pi_{V_k}(w-z_k)-z_k| \leq \eps c_k^{-1}
\]
for large enough $k$.  Moreover, since $V_k \to V$ in $d_G$, for sufficiently large $k$ we have
\[
|\pi_{V}(w-z_k)+z_k - \pi_{V_k}(w-z_k)-z_k)| \leq 2d_G(V_k, V) |w-z_k| \leq \eps c_k^{-1}.
\]
Finally, $\pi_{V}(w-z_k)+z_k = S^{-1}_k \pi_V  S_k (w)$ and the result follows.
\end{proof}

Next we provide a pair of simple algebraic  identities.

\begin{lma} \label{algebra}
For all integers $k$  and  all $w  \in \rd$
\begin{equation} \label{alg1}
S_k\pi_VS_k^{-1}(w) = \pi_V(w) -\pi_V(t_k)+t_k
\end{equation}
and
\begin{equation} \label{alg2}
S_k \pi_V  (w) +\pi_V (t_k)-2t_k    = c_k  S^{-1}_k \pi_V  S_k  (w) .
\end{equation}
\end{lma}

\begin{proof}
These identities follow immediately by applying the definition of $S_k$ and using the fact that linear homotheties and orthogonal projections commute. 
\end{proof}

We are now ready to prove Proposition \ref{keylemma}

\begin{proof}
Fix $x \in \pi_V(E)$.  Choose $k$ large enough to  guarantee that the conclusion of Lemma \ref{lma1} holds and also  that
\begin{equation} \label{tangentclose}
\rho_\mathcal{H} (E, S_k(F)) \leq  \eps/2.
\end{equation}
Choose $y \in  S_k(F) \cap B(0,1)$ such that
\begin{equation} \label{yclose}
|x - \pi_V(y) | \leq \eps
\end{equation}
which we may do by first applying \eqref{tangentclose} and then the fact that orthogonal projections do not increase distances. Then
\begin{eqnarray*}
| x - U_k \Pi_t S_k^{-1}(y) | & = & | x - S_k U_k^0 \Pi_t S_k^{-1}(y) |   \\ \\
& \leq  & | x - S_k \pi_V  S_k^{-1} (y) -\pi_V (t_k)+t_k |  \\ \\
&\, & \qquad + \  | S_k \pi_V  S_k^{-1} (y) +\pi_V (t_k)-t_k   - S_k U_k^0 \Pi_t S_k^{-1}(y) |   \\ \\
& =  & | x -   \pi_V   (y)  |  \hspace{8cm}  \text{by \eqref{alg1}}   \\ \\
&\, & \qquad + \  | S_k \pi_V  S_k^{-1} (y) +\pi_V (t_k)-2t_k    - c_k U_k^0 \Pi_t S_k^{-1}(y)|   \\ \\
& =  & | x -   \pi_V   (y)  |  \ + \ c_k | S^{-1}_k \pi_V  S_k  S_k^{-1} (y)    -  U_k^0 \Pi_t S_k^{-1}(y) |  \qquad \text{by \eqref{alg2}}   \\ \\
& \leq   & \eps  \ + \  c_k ( 2 c_k^{-1} \eps)
\end{eqnarray*}
by \eqref{yclose} and Lemma \ref{lma1}. Since $S_k^{-1}(y) \in F \cap B(z_k, c_k^{-1}) \subseteq F$, we have proved that, for all sufficiently large $k$,
\[
\rho_\mathcal{H}( \pi_V  (E), U_k \Pi_t F) \leq 3 \eps.
\]
Since, by Lemma \ref{uk},  $U_k$ satisfies the conditions required in Definition \ref{weaktangentdef} for sufficiently large $k$, it follows that  $ \pi_V  (E)$ is a   tangent to $\Pi_t(F)$, completing the proof.
\end{proof}

\subsection{Proof of Theorem \ref{thmmain}} \label{proof1}

Theorem \ref{thmmain} follows succinctly from Proposition  \ref{keylemma}.   First suppose $F$ is closed.  Apply Theorem \ref{goodweaktangent} to obtain a  simple   tangent $E$ with  focal point $z \in F$ satisfying $\hd E = \ad F$.   Proposition \ref{keylemma},  the differentiability assumption in Definition \ref{defproj}, and Theorem \ref{weaktangent} imply that for $\mu$ almost all $t \in \Omega$
\[
\ad \Pi_t(F) \geq \ad \pi_{V(t)}(E) 
\]
for $V(t) = T_z(t)  = \textup{ker}(J_{z} \Pi_t)^\perp  \in G(n,m)$.    Since
\[
\ad \pi_{V}(E) \geq \underset{V \sim \mathbb{P}}{\textup{essinf}} \  \ad \pi_V (E)
\]
 for $\mathbb{P}$ almost all $V \in G(n,m)$ and $\mu \circ T_z^{-1} \ll \mathbb{P}$ (the absolute continuity assumption in Definition \ref{defproj}), it follows that
\[
\ad \pi_{V(t)}(E) \geq  \underset{V \sim \mathbb{P}}{\textup{essinf}} \  \ad \pi_V (E)
\]
holds for $\mu$ almost all  $t \in \Omega$.  Therefore, since $\hd E = \ad F$,
\[
\ad \Pi_t(F) \geq  \inf_{\substack{E \subseteq \rd :\\ \hd E = \ad F}}
    \underset{V \sim \mathbb{P} }{\textup{essinf}} \  \ad \pi_V (E) 
\]
holds for $\mu$ almost all  $t \in \Omega$, proving the theorem for closed $F$.  However, if $F$ is not closed, then $\overline{\Pi_t(F)} \supseteq \Pi_t(\overline{F})$ since $\Pi_t$ is continuous.  Therefore, since Assouad dimension is stable under taking closure,
\[
\ad \Pi_t(F) = \ad \overline{\Pi_t(F)} \geq  \ad \Pi_t(\overline{F}) 
\]
and the desired result follows by applying the result for closed sets.

\section{Proof of Theorem \ref{distmain}} \label{distproof}

A key step in the proof of Theorem \ref{distmain} will be to relate pinned distance sets and radial projections via radial product sets.  Given $X \subseteq S^{n-1}$ and $Y \subseteq \mathbb{R}$, we define the \emph{radial product} of $X$ and $Y$ to be the set
\[
X \otimes Y = \{ xy : x \in X, y \in Y\} \subseteq \rd.
\]
The following is more general than we need.  We write $\ubd$ for the upper box dimension and note that for bounded sets $E \subseteq \rd$
\[
\hd E \leq \ubd E \leq \ad E.
\]
For concreteness, the upper box dimension of a bounded set $E$ is the infimum of $\alpha>0$ such that there is a constant $C\geq 1$ such that, for all $r>0$, $E$ may be covered by fewer than $Cr^{-\alpha}$ sets of diameter $r$.
\begin{lma} \label{radprod}
For  $X \subseteq S^{n-1}$ and bounded $Y \subseteq \mathbb{R}$,
\[
\hd (X \otimes Y) \leq \hd X + \ubd Y.
\]
\end{lma}
\begin{proof}
This is straightforward but we include the details due to its importance. Fix $s>\hd X$ and $t>\ubd Y$.  Let $\eps>0$, $\delta>0$ and  $\{U_i\}_i$ be a finite or countable  $\delta$-cover of $X$ such that
\[
\sum_i |U_i|^s \leq \eps.
\]
Consider the `wedge' $W_i = \{xy : x \in X \cap U_i,  y \in Y\}$.  By the definition of upper box dimension, there exists a uniform constant $C \geq 1$ such that $W_i$ may be covered by fewer than $C|U_i|^{-t}$  sets of diameter $|U_i|$.  Taking the union of these covers over all $i$ yields a $\delta$-cover $\{V_j\}_j$ of $X \otimes Y$ satisfying
\[
\sum_j |V_j|^{s+t} \leq  \sum_i |U_i|^{s+t} C|U_i|^{-t} \leq C \eps
\]
which proves that $\hd (X \otimes Y) \leq s+t$, and thus the lemma.
\end{proof}

It is immediate that for all sets $E \subseteq \rd$ and $z \in \rd$
\begin{equation} \label{radrel}
E \subseteq (\pi_z(E)-z) \otimes D_z(E)+z.
\end{equation} 
Indeed, for $x \in E$
\[
(\pi_z(x)-z) \otimes D_z(x)+z = x.
\]
Therefore Lemma \ref{radprod} yields
\begin{equation} \label{radkey}
\hd E \leq \hd \pi_z(E) + \ubd  D_z(E).
\end{equation}
We are now ready to prove Theorem \ref{distmain}.
\begin{proof}
It was proved in \cite[Lemma 3.1]{fraserdist} that if $F \subseteq \rd$ is a closed set and $E$ a simple tangent to $F$, then
\[
\ad D(F) \geq \ad D(E).
\]
Therefore it is sufficient to work with tangents of $F$.  Assume for now that $F$ is closed and apply Theorem \ref{goodweaktangent} to obtain a compact simple  tangent $E$ to $F$ with
\[
\hd E = \ad F.
\]
Apply Theorem \ref{goodweaktangent} a second time  to obtain a compact simple  tangent $E'$ to $E$ with
\[
\hd E'  = \ad E = \ad F
\]
and let $z \in E$ be the focal point of $E'$.    Let $\mathcal{E} \subseteq G(n,1)$ be the set of exceptions to  \eqref{assouadprojection}  applied to $E'$.  By definition $\mathcal{E}$ has  Hausdorff dimension at most $\theta=\theta(\ad F,n,1)$, recall \eqref{eta}. We now split into two cases.
\\ \\ \newpage
Case 1: Suppose $\hd \pi_z(E) >\theta$.  Since $\hd \mathcal{E} \leq \theta$, there must exist $x \in E$ such that $\textup{span}(z-x) \notin \mathcal{E}$.  Proposition \ref{keylemma} implies that $\pi_{\textup{span}(z-x)}(E')$ is a tangent to $D_z(E)$.  Therefore, applying Theorem \ref{weaktangent},
\begin{eqnarray*}
\ad D(F) \geq \ad D(E) \geq \ad D_z(E) \geq \ad \pi_{\textup{span}(z-x)}(E') &\geq& \min\{\ad E', \, 1\} \\ 
&=& \min\{\ad F, 1\}.
\end{eqnarray*}
Case 2: Suppose $\hd \pi_z(E) \leq \theta$.  It follows from \eqref{radkey} that
\[
\hd E \leq \hd \pi_z(E) + \ubd  D_z(E) \leq \theta + \ad D_z(E)
\]
and therefore
\[
\ad D(F) \geq \ad D(E) \geq \ad D_z(E) \geq \hd E - \theta = \ad F - \theta.
\]
Therefore we have proved the desired result for closed sets $F$.   If $F$ is not closed,  then $\overline{D(F)} \supseteq D(\overline{F})$ since the  map from $\rd \times \rd$ to $\mathbb{R}$ defined by $(x,y) \mapsto |x-y|$ is continuous.  Therefore, since Assouad dimension is stable under taking closure,
\[
\ad D(F) = \ad \overline{D(F)} \geq  \ad D(\overline{F}) 
\]
and the desired result follows by applying the result for closed sets.
\end{proof}

\subsection{Extension to   general  norms}

The proof given in Section \ref{distproof} goes through almost verbatim if the distance set is defined via a general norm $\| \cdot \|$ with the property that the boundary of the unit ball $\partial B$ is a $C^1$ manifold and the associated Gauss map cannot decrease Hausdorff dimension, see Section \ref{norm}. In the definition of radial product, $S^{n-1}$ is replaced by $\partial B$ and then \eqref{radrel} holds with the radial projection and pinned distance maps taken with respect to the norm $\| \cdot \|$.   The proof of \cite[Lemma 3.1]{fraserdist} goes through almost unchanged in the setting of general norms and therefore we can reduce to tangents $E$ and $E'$ in exactly the same way. Finally,  writing   $\Pi_t^{\| \cdot \|}: \rd \to \mathbb{R}$ for  the pinned distance map with respect to  $\| \cdot \|$,  the case 1 assumption $\hd \pi_z(E) >\theta$ still guarantees existence of $x \in E$ such that $T_z(x) =  \textup{ker}(J_z \Pi_x^{\| \cdot \|})  \notin \mathcal{E}$.  This is because the restriction of $T_z$ to  $(\partial B+z)$   coincides with the Gauss map $g : (\partial B+z) \to S^{n-1}$  (upon identification of antipodal points in $S^{n-1}$ and then identification  with $G(n,1)$)  and we assume the Gauss map cannot decrease Hausdorff dimension.

\section{Proof of Theorem \ref{sumset}} \label{sumsetproof}

Apply Theorem \ref{goodweaktangent} to obtain a simple tangent $E'$ to $E$  with $\hd E'  = \ad E$ and let  $z \in E$ be the focal point of $E'$.   It is straightforward to see that $F_\lambda$   is itself a simple tangent to $F_\lambda$ with focal point 0.  Therefore $E' \times F_\lambda$ is a simple tangent to $E \times F_\lambda$ with focal point $(z,0)$ for all $\lambda \in (0,1/2)$.  Let $V = \textup{span}(1,1) \in G(2,1)$.  It follows from Proposition \ref{keylemma} that $\pi_V (E' \times F_\lambda)$ is a tangent to $\pi_V(E \times F_\lambda)$ and therefore, by Theorem \ref{weaktangent},
\[
\ad (E + F_\lambda) = \ad \pi_V(E \times F_\lambda) \geq \hd \pi_V (E' \times F_\lambda) = \hd (E' + F_\lambda) 
\]
for all $\lambda \in (0,1/2)$. It follows from \cite[Theorem 5.12]{peresschlag} that
\[
\hd (E' + F_\lambda)  = \min\{\hd E'+\hd F_\lambda, \, 1\} =  \min\{\ad E+\ad F_\lambda, \, 1\}
\]
for almost all $\lambda \in (0,1/2)$ and even all $\lambda \in (a,b) \subseteq (0,1/2)$ outside of a set of exceptions of Hausdorff dimension at most $\hd E'+\hd F_b = \ad E+\ad F_b$, completing the proof.

\begin{samepage}

\subsection*{Acknowledgements}

The  author was   supported by an \emph{EPSRC Standard Grant} (EP/R015104/1) and a  \emph{Leverhulme Trust Research Project Grant} (RPG-2019-034).

\end{samepage}

\vspace{10mm}

\begin{samepage}

\noindent \emph{Jonathan M. Fraser\\
School of Mathematics and Statistics\\
The University of St Andrews\\
St Andrews, KY16 9SS, Scotland} \\
\noindent  Email: jmf32@st-andrews.ac.uk\\ \\

\end{samepage}


\begin{thebibliography}{99}

\bibitem[B17]{barany}
B. B\'ar\'any.
On some non-linear projections of self-similar sets in $\mathbb{R}^3$,
\emph{Fund. Math.},  {\bf 237}, (2017), 83--100.


\bibitem[BLZ16]{zahl}
M. Bond, I. \L{}aba and J. Zahl.
Quantitative visibility estimates for unrectifiable sets in the plane,
\emph{Trans. Amer. Math. Soc.},  {\bf 368}, (2016), 5475--5513.


\bibitem[F85]{distancesets}
K. J. Falconer. On the Hausdorff dimensions of distance sets, \emph{Mathematika}, {\bf 32}, (1985), 206--212.


\bibitem[F14a]{falconer}
K. J. Falconer.
{\em Fractal Geometry: Mathematical Foundations and Applications},
 John Wiley \& Sons, Hoboken, NJ, 3rd. ed., 2014.

\bibitem[F04]{poly}
K. J. Falconer.
Dimensions of intersections and distance sets for polyhedral norms, \emph{Real Anal. Ex.}, {\bf 30}, (2004), 719--726.



\bibitem[FFJ15]{FalconerFraserJin}
K.~J. Falconer, J. M. Fraser and X. Jin.
Sixty Years of Fractal Projections,
\emph{Fractal geometry and stochastics V, (Eds. C. Bandt, K.~J. Falconer  and M. Z\"ahle)}, Birkh\"auser, Progress in Probability, (2015).

\bibitem[F14b]{fraserassouad}
J.~M. Fraser.
 Assouad type dimensions and homogeneity of fractals,
 {\em Trans. Amer. Math. Soc.}, {\bf 366}, (2014), 6687--6733.

\bibitem[F18]{fraserdist}
J. M. Fraser.
Distance sets, orthogonal projections, and passing to weak tangents, 
   \emph{Israel J.  Math.}, {\bf 226}, (2018), 851--875.

\bibitem[F20]{fraserbook}
J. M. Fraser.
{\em Assouad Dimension and Fractal Geometry},
 Cambridge University Press, Tracts in Mathematics Series, \emph{in press}.

\bibitem[FHOR15]{fraserrobinson}
J.~M. Fraser, A. M. Henderson, E. J. Olson and J. C. Robinson.
On the Assouad dimension of self-similar sets with overlaps, 
  \emph{Adv. Math.}, {\bf 273}, (2015), 188--214.

\bibitem[FK]{kaenmakiproj}
J. M. Fraser  and A. K\"aenm\"aki.
Attainable values for the Assouad dimension of projections,
\emph{ Proc. Amer. Math. Soc.}, to appear,  available at:  https://arxiv.org/pdf/1811.00951



\bibitem[FO17]{FraserOrponen}
J.~M. Fraser and T. Orponen.
The Assouad dimensions of projections of planar sets,
 {\em Proc. Lond. Math. Soc.}, {\bf 114}, (2017), 374--398.

\bibitem[G02]{topology}
M. Ghomi.
Gauss map, topology, and convexity of hypersurfaces with nonvanishing curvature,
\emph{Topology}, {\bf  41}, (2002), 107--117.

\bibitem[GIOW20]{guth}
L. Guth, A. Iosevich, Y. Ou and H. Wang. 
On Falconer’s distance set problem in the plane, \emph{Invent. Math.}, {\bf 219},  (2020), 779--830.

\bibitem[HS12]{hochmanshmerkin}
M. Hochman and P. Shmerkin.
Local entropy averages and projections of fractal measures,
\emph{Ann. Math.}, {\bf 175}, (2012), 1001--1059.

\bibitem[KOR18]{anti3}
  A.~K{\"a}enm{\"a}ki, T.~Ojala, and E.~Rossi.
Rigidity of   quasisymmetric mappings on self-affine carpets, \emph{Int. Math. Res. Not. IMRN}, \textbf{12}, (2018), 3769--3799.

\bibitem[KS19]{keletishmerkin}
T. Keleti and P. Shmerkin.
New bounds on the dimensions of planar distance sets,
\emph{Geom. Funct. Anal.},  {\bf 29}, (2019), 1886--1948.

\bibitem[L98]{luk}
J. Luukkainen.
Assouad dimension: antifractal metrization, porous sets, and homogeneous measures,
\emph{J. Korean Math. Soc.}, {\bf 35}, (1998), 23--76.

\bibitem[MT10]{mackaytyson}
J. M. Mackay and J. T. Tyson.
\emph{Conformal dimension: Theory and application},
University Lecture Series, {\bf 54}, American Mathematical Society, Providence, RI, (2010).

\bibitem[M54]{marstrand}
J.~M. Marstrand.
 Some fundamental geometrical properties of plane sets of fractional dimensions,
{\em Proc. London Math. Soc.(3)}, {\bf 4}, (1954), 257--302.


\bibitem[M75]{mattilaproj}
P.~Mattila.
 Hausdorff dimension, orthogonal projections and intersections with planes,
 {\em Ann. Acad. Sci. Fenn. A Math.} {\bf 1} (1975),  227--244.

\bibitem[M95]{mattila}
P. Mattila.
{\em Geometry of sets and measures in Euclidean spaces},
 Cambridge studies in advanced mathematics, \textbf{44}, Cambridge University Press, (1995).


\bibitem[M14]{MattilaSurvey}
P. Mattila.
Recent progress on dimensions of projections,
in {\em Geometry and Analysis of Fractals}, D.-J. Feng and K.-S. Lau (eds.), pp 283--301,
{\it  Springer Proceedings in Mathematics \& Statistics.} {\bf 88}, Springer-Verlag, Berlin Heidelberg, (2014).





\bibitem[O17]{orponenAD}
T. Orponen.
On the distance sets of Ahlfors-David regular sets, 
\emph{Adv. Math.}, {\bf 307}, (2017), 1029--1045.


\bibitem[O19]{orponenradial}
T. Orponen.
On the dimension and smoothness of radial projections, 
\emph{Anal. PDE}, {\bf 12}, (2019), 1273--1294.





\bibitem[O]{orponenassouad}
T. Orponen.
On the Assouad dimension of projections,
\emph{Proc. Lond. Math. Soc.}, to appear,  available at:  https://arxiv.org/abs/1902.04993


\bibitem[PS00]{peresschlag}
Y. Peres and W. Schlag.
Smoothness of projections, Bernoulli convolutions, and the dimension of exceptions, 
\emph{Duke Math. J.}, {\bf 102}, (2000), 193--251.







\bibitem[R11]{robinson}
J. C. Robinson.
{\em Dimensions, Embeddings, and Attractors},
Cambridge University Press, 2011.

\bibitem[S17]{pablo}
P. Shmerkin.
On distance sets, box-counting and Ahlfors-regular sets,
\emph{Discrete Analysis}, {\bf 9}, (2017).


\bibitem[S19]{pablo2}
P. Shmerkin.
On the Hausdorff dimension of pinned distance sets,
\emph{Israel J. Math.}, {\bf 230}, (2019),  949--972.


\bibitem[S]{pablo3}
P. Shmerkin.
A nonlinear version of Bourgain's projection theorem,
preprint, available at: https://arxiv.org/abs/2003.01636


\end{thebibliography}
\end{document}